\documentclass{article}

\usepackage{makeidx}
\usepackage{graphicx}
\usepackage[english,activeacute]{babel}
\usepackage[latin1]{inputenc}
\usepackage{amsmath}
\usepackage{amsfonts}
\usepackage{amssymb,latexsym}
\usepackage[T1]{fontenc}
\usepackage{array}
\usepackage{booktabs}
\usepackage{graphicx}
\usepackage{color}
\usepackage{todonotes}

\newtheorem{theorem}{Theorem}
\newtheorem{corollary}{Corollary}
\newtheorem{proposition}{Proposition}
\newtheorem{remark}[theorem]{Remark}
\newenvironment{proof}[1][Proof]{\noindent\textbf{#1.} }{\ \rule{0.5em}{0.5em}}

\def\pFq#1#2#3#4#5{%
	{}_{#1}F_{#2}\biggl(\genfrac..{0pt}{}{#3}{#4};#5\biggr)%
}
\begin{document}

\title{Asymptotics of orthogonal polynomials generated by a Geronimus
perturbation of the Laguerre measure}
\author{Alfredo Dea\~no$^{1}$, Edmundo J. Huertas$^{2}$, Pablo Rom\'an$^{3}$ 
\\[2mm]
$^{1}$Departamento de Matem\'aticas\\
Universidad Carlos III de Madrid, Madrid, Spain\\
alfredo.deanho@uc3m.es\\
[3mm] $^{2}$ Dept. de Ingenier\'{i}a Civil: Hidr\'{a}ulica, Energ\'{i}a y Medio Ambiente%
\\
ETSI Caminos, Canales y Puertos\\
Universidad Polit\'{e}cnica de Madrid, Spain\\
ej.huertas.cejudo@upm.es, ehuertasce@gmail.com\\
[3mm] $^{3}$ CIEM, FaMAF, Universidad Nacional de C\'ordoba,\\
C\'ordoba, Argentina\\
roman\symbol{'100}famaf.unc.edu.ar}
\date{\today }
\maketitle

\begin{abstract}
This paper deals with monic orthogonal polynomials generated
by a Geronimus canonical spectral transformation of the Laguerre classical
measure, i.e., 
\begin{equation*}
\frac{1}{x-c}x^{\alpha }e^{-x}dx+N\delta (x-c),
\end{equation*}%
for $x\in[0,\infty)$, $\alpha>-1$, a free parameter $N\in \mathbb{R}_{+}$\ and a
shift $c<0$. We analyze the asymptotic behavior (both strong and relative 
to classical Laguerre polynomials) of these orthogonal polynomials as $n$ tends to infinity. 

\smallskip

\textbf{AMS Subject Classification:} 33C45, 41A60, 33C15

\textbf{Keywords:} Orthogonal polynomials, Canonical spectral
transformations of measures, Asymptotic analysis, Hypergeometric functions. 
\end{abstract}

\section{Introduction}
Let $\mu $ be a positive Borel measure supported on a finite or infinite
interval $E=\mathrm{supp}(\mu )$, such that the convex hull verifies $%
C_{0}(E)=[a,b]\subseteq \mathbb{R}$. In the last years several authors have
studied the so called \textit{canonical spectral transformations} of $\mu $,
which are a way to construct new families of orthogonal polynomials from a
perturbed version of $\mu $. They have been studied from several points of
view, including the corresponding Jacobi matrices (see \cite{BDT-NA10}, \cite%
{Y-BKMS02}) or the Stieltjes functions associated with such a kind of
transformations (see \cite{Z-JCAM97} among others).

Let us introduce the sequences of monic orthogonal polynomials (SMOP in the
sequel) associated with one of the aforesaid canonical transformations,
called the \textit{Geronimus canonical transformation on the real line}. The
basic \textit{Geronimus perturbation} of $\mu $ is defined as%
\begin{equation*}
\frac{1}{x-c}d\mu (x)+N\delta (x-c),  
\end{equation*}%
where $N\in \mathbb{R}_{+}$, $\delta (x-c)$ is the Dirac delta function
located at $x=c$, and the shift of the perturbation verifies $c\not\in E$.
Observe that it is given simultaneously by a rational modification of $\mu $
by a positive linear polynomial whose real zero $c$ is the point of
transformation (also known as the \textit{shift} of the transformation)
jointly with the addition of a Dirac mass at the point of transformation as
well.

This transformation was introduced by Geronimus in his pioneer works \cite%
{G-HNeds40} and \cite{G-ZMOK40} on procedures of constructing new sequences
of orthogonal polynomials from other known sequences families, and it was
also studied by Shohat in a different scheme involving mechanical
quadratures (see \cite{S-TAMS37}). Years later, Maroni (see \cite{M-PMH90})
returned to the problem and gave a first expression of the Geronimus
perturbed orthogonal polynomials in terms of so called co-recursive
polynomials of the classical orthogonal polynomials. More recently, Bueno
and Marcell\'{a}n reinterpreted this perturbation in the framework of the
discrete Darboux transformations (see \cite{BM-LAA04}). In \cite{BDT-NA10}
the authors present a new computational algorithm for computing the
Geronimus transformation with large shifts. In \cite{BHR-LNCS14} the authors
provide sharp limits (and the speed of convergence to them) of the zeros of
the Geronimus perturbed SMOP, and also, when $\mu $ is semi-classical they
obtain the corresponding electrostatic model for the zeros of the Geronimus
perturbed SMOP, showing that they are the electrostatic equilibrium points
of positive unit charges interacting according to a logarithmic potential
under the action of an external field. In \cite{MV-JDEA14} the authors
extend the standard Geronimus transformation to a cubic case. \cite{DM-NA13}
provides a new revision of the Geronimus transformation in terms of
symmetric bilinear forms in order to include certain Sobolev and
Sobolev--type orthogonal polynomials into the scheme of Darboux
transformations. Finally, \cite{DGM-MultGer14} deals with multiple Geronimus
transformations and show that they lead to discrete (non-diagonal) Sobolev
type inner products, and it is shown that every discrete Sobolev inner
product can be obtained as a multiple Geronimus transformation.

In view of the foregoing, this transformation has been extensively studied
in the literature, mainly in analytic and algebraic frameworks. However, to
the best of our knowledge, the asymptotic properties of the family 
of orthogonal polynomials as $n\to\infty$ have not been studied in detail, 
save for the particular case when $N=0$ and the perturbed 
measure is the Laguerre classical measure (see \cite{Fej}).

\section{Laguerre polynomials and functions of the second kind}

The classical Laguerre polynomials $L_{n}^{\alpha}(x)$ are defined as the
polynomials orthogonal with respect to the $\mathrm{L}^{2}$ inner product%
\begin{equation*}
\langle p,q\rangle _{\alpha }=\int_{0}^{\infty }p(x)q(x)x^{\alpha
}e^{-x}dx,\quad \alpha >-1,\quad p,q\in \mathbb{P},
\end{equation*}%
see, among others \cite{Chi78} or \cite{Szego}.

In order to fix notation, we denote by $\widehat{L}_{n}^{\alpha}(x)$ 
the \emph{monic} Laguerre polynomials, so $\widehat{L}_{n}^{\alpha
}(x)=x^{n}+\ldots$. These monic polynomials are
connected to standard Laguerre polynomials $L_{n}^{\alpha}(x)$ by the
formula%
\begin{equation}
\widehat{L}_{n}^{\alpha }(x)=(-1)^{n}n!L_{n}^{\alpha}(x),\quad n\geq 0.
\label{PL}
\end{equation}%
They satisfy a three term recurrence relation that we write in the following
form%
\begin{equation}
x\widehat{L}_n^{\alpha}(x)=\widehat{L}_{n+1}^{\alpha }(x)+\beta _{n}\widehat{%
L}_{n}^{\alpha }(x)+\gamma _{n}\widehat{L}_{n-1}^{\alpha }(x),
\label{TTRRLaguerre}
\end{equation}%
where%
\begin{equation}
\beta _{n}=2n+\alpha +1,\qquad \gamma _{n}=n(n+\alpha ),
\label{bngnLaguerre}
\end{equation}%
and we have initial data $\widehat{L}_{0}^{\alpha }(x)=1$ and $\widehat{L}%
_{1}^{\alpha }(x)=x-\alpha -1$. We will also make use of the $\mathrm{L}^{2}([0,\infty))$ norm of
the monic Laguerre polynomials. Since%
\begin{equation*}
\Vert L_{n}^{\alpha}\Vert _{\alpha }^{2}=\frac{\Gamma (n+\alpha +1)}{%
\Gamma (n+1)},
\end{equation*}%
we have%
\begin{equation*}
\Vert \widehat{L}_{n}^{\alpha }\Vert _{\alpha }^{2}=\Gamma (n+\alpha
+1)\Gamma (n+1).  
\end{equation*}

A second (independent) solution of the recurrence relation (\ref%
{TTRRLaguerre}) is the function of the second kind, obtained via a Stieltjes
transform of the Laguerre polynomials%
\begin{equation}
\widehat{F}_{n}^{\alpha }(z)=\int_{0}^{\infty }\frac{\widehat{L}_{n}^{\alpha
}(t)}{t-z}t^{\alpha }e^{-t}dt,  \label{F}
\end{equation}%
which is an analytic function for $z\in \mathbb{C}\setminus \lbrack 0,\infty
)$. Using the Rodrigues formula for Laguerre polynomials \cite[§18.5(ii)]%
{DLMF} and a standard integral representation, see \cite[Eq.13.4.4]{DLMF}
for instance, it is possible to write $\widehat{F}_{n}^{\alpha }(z)$ in
terms of the confluent hypergeometric function of the second kind, or Kummer 
$U$ function: 
\begin{equation}
\widehat{F}_{n}^{\alpha }(z)=(-1)^{n}n!\Gamma (n+\alpha +1)U(n+1,1-\alpha
,ze^{\pm \pi i}),  \label{defF}
\end{equation}%
with plus sign if $-\pi<\textrm{arg}\, z\leq 0$ and minus sign if $0<\textrm{arg}\, z\leq \pi$. 
This representation will be a key element for all the asymptotic analysis
later on in this manuscript. For more
information about the confluent hypergeometric functions, we refer the
reader for instance to \cite[Chapter 13]{DLMF}.

Let us introduce the following inner product in the linear space $\mathbb{P}$%
\ of polynomials with real coefficients%
\begin{equation}
\langle p,q\rangle _{\nu _{N}}=\int_{0}^{\infty }p(x)q(x)\frac{x^{\alpha}e^{-x}}{%
x-c}dx+N\delta (x-c),  \label{Qn}
\end{equation}%
where $\alpha >-1$, and $N\geq 0$, and $c\in (-\infty ,0)$. Namely, we deal
with a measure that consists of an absolutely continuous part, which is a
rational perturbation of the Laguerre weight on $[0,+\infty )$, plus a Dirac
delta located at point $x=c$: 
\begin{equation*}
d\nu _{N}(x)=\frac{x^{\alpha }e^{-x}}{x-c}dx+N\delta (x-c).  
\end{equation*}%
Equivalently, we will say that $\nu _{N}$\ is a Geronimus perturbation of
the standard Laguerre measure (see \cite{BDT-NA10},\cite{BHR-LNCS14} and the
references therein), and we will denote by $\widehat{Q}_{n}^{\alpha ,c,N}(x)$
the monic orthogonal polynomials with respect to (\ref{Qn}).

As explained in \cite{BHR-LNCS14} (and the references therein), the \emph{%
Laguerre--Geronimus monic orthogonal polynomials} can be written in terms of the
monic Laguerre OPs using the following simple connection formula: 
\begin{equation}
\widehat{Q}_{n}^{\alpha ,c,N}(x)=\widehat{L}_{n}^{\alpha }(x)+\Lambda
_{n}^{N}\widehat{L}_{n-1}^{\alpha }(x),  \label{connection}
\end{equation}%
where the coefficient $\Lambda _{n}^{N}$ depends on $n$, $\alpha $, $c$ and $%
N$. It can be written as 
\begin{equation}
\Lambda _{n}^{N}=-\frac{\Gamma (n+\alpha )\Gamma (n)}{L_{n-1}^{\alpha
}(c)F_{n-1}^{\alpha }(c)-NL_{n-1}^{\alpha }(c)^{2}}-\pi _{n-1}(c),
\label{Lambdan1}
\end{equation}%
where we have defined, for $c\in(-\infty ,0)$, 
\begin{equation}
\pi _{n}(c)=\frac{\widehat{L}_{n+1}^{\alpha }(c)}{\widehat{L}_{n}^{\alpha
}(c)}.  \label{pi_n}
\end{equation}

We observe that in Proposition 1 and Remark 1 in \cite{BHR-LNCS14},
different expressions for $\Lambda _{n}^{N}$ are given, but using identities
for the Laguerre polynomials it is not difficult to obtain (\ref{Lambdan1})
from those.

A particular case is given by $N=0$, where the Dirac delta is not present in
the perturbed measure. Then $\Lambda _{n}^{N}=-r_{n-1}(c)$, where 
\begin{equation}
r_{n}(c)=\frac{\widehat{F}_{n+1}^{\alpha }(c)}{\widehat{F}_{n}^{\alpha }(c)}%
,\quad c\in (-\infty ,0),  \label{r_n}
\end{equation}%
with $F_{-1}^{\alpha }(c)=1$, analogously to \cite[\S 2.4.4]{Gauts04}.

The aim of this paper is to obtain strong and relative asymptotics of the
sequence of OPs $\widehat{Q}_{n}^{\alpha ,c,N}(z)$, for large $n$ and other parameters fixed. 
The simplicity of the connection formula (\ref{connection}) makes it a very
attractive identity to use in conjunction with classical asymptotic
approximations for Laguerre polynomials (such as Perron, Fej\'er or
Mehler-Heine expansions, see \cite[\S 8.22]{Szego}), in order to obtain the
corresponding result for the Laguerre--Geronimus OPs. The only element that is missing
so far in the literature is a study of the asymptotic behavior of the
coefficient $\Lambda _{n}^{N}$. We observe that because of (\ref{Lambdan1}), $\Lambda_n^N$
depends on the ratios $\pi _{n}(c)$ and $r_{n}(c)$.

The structure of the paper is as follows: in Section \ref{Sect_asymp_pi_r}
we obtain large $n$ asymptotic expansions for $\pi _{n-1}(z)$ and $%
r_{n-1}(z)$ in the complex plane, which lead to asymptotic approximations 
for $\Lambda _{n}^{N}$ in Section \ref{Sect_asymp_Ln}. Putting together this result and the
connection formula (\ref{connection}), we obtain in Section \ref%
{Sect_asymp_Q} the strong and relative asymptotics for 
$\widehat{Q}_{n}^{\alpha ,c,N}(z)$ in the complex plane. 
Finally, Section \ref{Sect_asymp_rc} is devoted to the 
study of the coefficients in the three term recurrence relation satisfied by 
$\widehat{Q}_{n}^{\alpha ,c,N}(z)$, and Section \ref{Sect_hypergeom} gives 
a hypergeometric representation for the OPs that could be of independent interest.


\section{Asymptotic expansions for $\protect\pi _{n-1}$ and $r_{n-1}$}
\label{Sect_asymp_pi_r} 

The ratios $\pi _{n-1}(z)$ and $r_{n-1}(z)$ could in
principle be studied using standard techniques for the asymptotic behavior of
solutions of three-term recurrence relations, such as the Perron theorem,
see for instance \cite[§4.3]{GST}. However, since the recurrence
coefficients in (\ref{TTRRLaguerre}) satisfy $\beta_n\sim 2n$ and $%
\gamma_n\sim n^2$ as $n\to\infty$, the theorem is inconclusive about the
existence of minimal and dominant solutions, and it does not give detailed
asymptotic information about the behavior of ratios of solutions. We refer
the reader to \cite[Section 4]{DST2008} for more details.

In this paper we work with strong asymptotics of the Laguerre polynomials
and functions of the second kind directly. For $z$ away from $[0,\infty)$,
the strong asymptotics for the $L_n^{\alpha}(z)$ can be obtained from the
classical expansion due to Perron, see for instance \cite[Theorem 8.22.3]%
{Szego}: 
\begin{equation}
L_{n}^{\alpha}(z)=\frac{1}{2\sqrt{\pi }}e^{z/2}(-z)^{-\frac{\alpha }{2}-%
\frac{1}{4}}n^{\frac{\alpha }{2}-\frac{1}{4}}e^{2\sqrt{-nz}}\left( 1+%
\mathcal{O}(n^{-1/2})\right) ,  \label{Perron_Laguerre}
\end{equation}%
which is valid for fixed $\alpha>-1$ and $z\in \mathbb{C}\setminus \lbrack
0,\infty )$. The fractional powers are assumed to take their principal values, 
with phase between $-\pi$ and $\pi$. In \cite[Theorem 3]{DHM} higher terms in this asymptotic
expansion have been obtained, using a related expansion for confluent
hypergeometric functions due to Buchholz, see also \cite{LopTem} and references
therein. The ratio asymptotics is given in \cite{DHM} as well: 
\begin{eqnarray*}
\frac{L_{n}^{\alpha}(z)}{L_{n-1}^{\alpha}(z)} &=&1+\sqrt{-\frac{z}{n-1}%
}+\frac{2\alpha -2z-1}{4(n-1)}+\mathcal{O}(n^{-3/2})  \notag \\
&=&1+\sqrt{-\frac{z}{n}}+\frac{2\alpha -2z-1}{4n}+\mathcal{O}(n^{-3/2}),
\end{eqnarray*}%
as $n\rightarrow \infty $, for fixed $\alpha $ and $z\in\mathbb{C}%
\setminus[0,\infty)$. Therefore, as a direct consequence, for the monic
polynomials we have 
\begin{equation}
\pi _{n-1}(z)=\frac{\widehat{L}_{n}^{\alpha }(z)}{\widehat{L}_{n-1}^{\alpha
}(z)}=-n\frac{L_{n}^{\alpha }(z)}{L_{n-1}^{\alpha}(z)}=-n-\sqrt{-zn}+%
\frac{2z-2\alpha +1}{4}+\mathcal{O}(n^{-1/2}),  \label{asymppi}
\end{equation}%
as $n\rightarrow \infty $.

Regarding the asymptotic behavior of the functions of the second kind, we
present the following result:

\begin{proposition}
\label{Prop_asympF} Given fixed $c\in (-\infty ,0)$ and $\alpha >-1$, the
functions of the second kind $\widehat{F}_{n}^{\alpha }(z)$, defined by (\ref%
{F}), satisfy 
\begin{equation}
\begin{aligned} \widehat{F}_n^{\alpha}(z)&= (-1)^n \sqrt{\pi
}(-z)^{\frac{\alpha }{2}-\frac{1}{4}}\,e^{-z/2-2\sqrt{-zn}}\,\Gamma (n+\alpha+1
)n^{-\frac{\alpha }{2}-\frac{1}{4}}\\ &\times \left[
e_{0}+\frac{e_{1}}{\sqrt{-zn}}+\frac{e_{2}}{-zn}+\mathcal{O}(n^{-3/2})%
\right] ,\qquad n\rightarrow \infty. \end{aligned}  \label{asympF}
\end{equation}
The expansion is valid for bounded $z\in\mathbb{C}\setminus [0,\infty)$,
with principal values of the power functions. The first few coefficients $%
e_j $ are 
\begin{equation*}
\begin{aligned} e_{0}(\alpha ,z)&=1, \label{eks} \\ e_{1}(\alpha
,z)&=\frac{12\alpha ^{2}-3-24z(1-\alpha )-4z^{2}}{48}, \\ e_{2}(\alpha
,z)&=\frac{16z^{4}+192(1-\alpha )z^{3}+24(20\alpha ^{2}-48\alpha
+13)z^{2}}{4608} \\ &+\frac{144(\alpha -1)(2\alpha +1)(2\alpha
+3)z+9(4\alpha ^{2}-1)(4\alpha ^{2}-9)}{4608}. 
\end{aligned}
\end{equation*}
\end{proposition}

We observe that the leading term in this expansion is consistent with the
results in \cite{Fej,MPP}, bearing in mind that we are working with monic
polynomials. We also note that the exponential factor is erroneously
corrected in Proposition 3.2 (a) in the first reference.

As a direct consequence, and using symbolic computation, we have an
asymptotic expansion for the ratio of consecutive functions of the second
kind:

\begin{proposition}
\label{Cor_asympF} As $n\rightarrow \infty $, the ratio asymptotics\ of the
Laguerre functions of the second kind is given by%
\begin{equation*}
r_{n-1}(z)=\frac{\widehat{F}_{n}^{\alpha }(z)}{\widehat{F}_{n-1}^{\alpha }(z)%
}=-n+\sqrt{-zn}+\frac{2z-2\alpha +1}{4}+\mathcal{O}(n^{-1/2}),
\end{equation*}
where $\alpha >-1$, $c\in (-\infty ,0)$ and $z\in\mathbb{C}%
\setminus[0,\infty)$.
\end{proposition}

We remark that division by $\widehat{F}_{n-1}^{\alpha }(z)$ is allowed for $z$ away
from the positive real axis, bearing in mind (\ref{defF}) and the fact that
the Kummer function $U(a,b,z)$ does not have zeros for $|\arg\, z|<\pi$ if $%
a $ is positive, see \cite[\S 13.9]{DLMF}.

In order to get the previous results, there are at least two possibilities:
use the expression of $\widehat{F}_{n}^{\alpha }(c)$ in terms of Kummer functions, see
(\ref{defF}), or deduce its asymptotic behavior from the Deift--Zhou
steepest descent method applied to the corresponding Riemann--Hilbert
problem, see the work of Vanlessen \cite{V-CA07} and the monograph by Deift 
\cite{Deift}. In the sequel, we elaborate on the first approach.

Following the ideas exposed in \cite{Temme81}, we use the integral
representation for the Kummer $U$-function: 
\begin{equation*}
U(a,b,z)=\frac{1}{\Gamma(a)}\int_0^{\infty} e^{-zt} t^{a-1}(1+t)^{b-a-1}dt,
\end{equation*}
which holds for $\text{Re}\,a,\text{Re}\, z>0$, see also \cite[13.4.4]{DLMF}. 
The transformation $t/(1+t)=e^{-\tau}$ gives 
\begin{equation*}
U(a,b,z)=\frac{e^{z/2}}{\Gamma(a)}\int_0^{\infty} e^{-a\tau-z/\tau}
\tau^{-b}f(\tau)d\tau,
\end{equation*}
where 
\begin{equation*}
f(\tau)=e^{z\mu(\tau)}\left(\frac{\tau}{1-e^{-\tau}}\right)^b, \qquad
\mu(\tau)=\frac{1}{\tau}-\frac{1}{e^{\tau}-1}-\frac{1}{2}.
\end{equation*}

The function $f(\tau)$ is analytic for $|\tau|<2\pi$, and therefore it
admits a power series expansion around the origin of the form 
\begin{equation}  \label{seriesf}
f(\tau)=\sum_{m=0}^{\infty} d_m(b,z)\tau^m.
\end{equation}

Integration term by term, invoking the classical Watson lemma, \cite{Olver,
Temme}, gives an expansion of the $U$-function of the form 
\begin{equation*}  
U(a,b,z)=\sum_{m=0}^{M-1} d_m(b,z)\phi_m(a,b,z)+R_M(a,b,z),
\end{equation*}
where the asymptotic sequence is 
\begin{equation*}  
\phi_m(a,b,z)=\frac{2\,e^{z/2}}{\Gamma(a)}\left(\frac{z}{a}\right)^{\frac{%
m+1-b}{2}}K_{m+1-b}(2(az)^{1/2}),
\end{equation*}
in terms of modified Bessel functions, using the fact that 
\begin{equation*}
K_{\nu}(2(z\zeta)^{1/2})=\frac{1}{2}\left(\frac{\zeta}{z}\right)^{-\nu/2}%
\int_0^{\infty}e^{-z\tau-\zeta/\tau} \tau^{-\nu-1}d\tau,
\end{equation*}
valid for $\text{Re}\,z,\text{Re}\, \zeta>0$. Also, we have 
$R_M(a,b,z)=\mathcal{O}(\phi_M(a,b,z))$ as $a\to\infty$, uniformly with respect to $z$ in compact
sets in $z\geq 0$ and uniformly with respect to $b$ in compact sets of $%
\mathbb{R}$. We assume, following \cite[\S 2.1]{Temme81} that $M$ is large
enough, in particular $M>b$.

If we replace $b=1-\alpha$, symbolic computation gives%
\begin{equation*}
\begin{aligned} d_{0}(\alpha,z)&=1, \label{cn} \\
d_{1}(\alpha,z)&=\frac{6(1-\alpha )-z}{12}, \\
d_{2}(\alpha,z)&=\frac{z^{2}-12(1-\alpha )z+12(\alpha -1)(3\alpha
-2)}{288}, \\ d_{3}(\alpha,z)&=\frac{-5z^{3}-90(1-\alpha )z^{2}-36(15\alpha
^{2}+25\alpha +8)z-1080\alpha (\alpha -1)^{2}}{51840}, 
\end{aligned}
\end{equation*}%
and so on, for the coefficients in the series expansion \eqref{seriesf}.
Replacing $a$ by $n+1$ in the asymptotic sequence, we obtain 
\begin{equation}
\phi _{m}(n,\alpha ,z)=\frac{2\,e^{z/2}}{n!}\left(\frac{z}{n+1}\right) ^{%
\frac{m+\alpha }{2}}K_{m+\alpha }(2((n+1)z)^{1/2}).  \label{phim2}
\end{equation}%
As a consequence, using (\ref{defF}) and (\ref{phim2}), we have 
\begin{equation*}
\begin{aligned} 
\widehat{F}_{n}^{\alpha}(z)
&=(-1)^{n}n!\,\Gamma (n+\alpha+1)\,U(n+1,1-\alpha ,ze^{\pm \pi i}) \\ 
&=2e^{-z/2}(-1)^{n}\Gamma (n+\alpha +1)\left[ S_{M}(n,\alpha
,ze^{\pm \pi i})+R_{M}(n,\alpha ,ze^{\pm \pi i})\right] , 
\end{aligned}
\end{equation*}
as $n\rightarrow \infty $, where 
\begin{equation*}
S_{M}(n,\alpha ,ze^{\pm \pi i})=\sum_{m=0}^{M-1}d_{m}(\alpha ,-z)\left( \frac{ze^{\pm \pi i}}{n+1}%
\right) ^{\frac{m+\alpha }{2}}K_{m+\alpha }(2((n+1)ze^{\pm \pi i})^{1/2}),  
\end{equation*}%
and $R_{M}(n,\alpha ,ze^{\pm \pi i})$ is the remainder. Once again, we take plus sign 
if $-\pi<\textrm{arg}\, z\leq 0$ and minus sign if $0<\textrm{arg}\, z\leq \pi$. 

It is possible
to re-expand this asymptotic series in inverse powers of $n$: using the
asymptotics of the modified Bessel functions for large values of the
argument and fixed order $\nu$, see \cite[10.40.2]{DLMF}, we have 
\begin{equation*}
K_{\nu }(z)\sim \left( \frac{\pi }{2z}\right) ^{1/2}e^{-z}\sum_{\ell
=0}^{\infty }\frac{a_{\ell }(\nu )}{z^{\ell }},\qquad z\rightarrow \infty ,
\end{equation*}%
where $a_{0}(\nu )=1$ and for $\ell \geq 1$,%
\begin{equation*}
a_{\ell }(\nu )=\frac{(4\nu ^{2}-1)(4\nu ^{2}-3^{2})\cdots (4\nu ^{2}-(2\ell
-1)^{2})}{8^{\ell }\ell !}.  
\end{equation*}%
So, if we denote $s_n=2((n+1)ze^{\pm \pi i})^{1/2}$, we have 
\begin{equation*}
K_{m+\alpha }(s_{n})\sim \left( \frac{\pi }{2s_{n}}\right)
^{1/2}e^{-s_{n}}\sum_{\ell =0}^{\infty }\frac{a_{\ell }(m+\alpha )}{%
s_{n}^{\ell }}.
\end{equation*}%
Note that $s_{n}=2\sqrt{nze^{\pm \pi i}}\,(1+\mathcal{O}(n^{-1}))$ as $n\rightarrow
\infty $. Assembling all the previous results and expanding in inverse
powers of $n$, we arrive at Proposition \ref{Prop_asympF}. Proposition \ref%
{Cor_asympF} follows then from this result, using symbolic computation to
manipulate the asymptotic expansions.


\section{Asymptotic behavior of $\Lambda _{n}^{N}$}
\label{Sect_asymp_Ln}

The main result of this section is the following:

\begin{proposition}
\label{Prop_asympL} Let $c\in (-\infty ,0)$, $\alpha >-1$ and $N\geq 0$ be
fixed parameters, then 
\begin{equation}
\Lambda _{n}^{N}=n+\sqrt{-cn}+\frac{2\alpha -2c-1}{4}+\mathcal{O}%
(n^{-1/2}),\qquad n\rightarrow \infty ,  \label{AsympLn}
\end{equation}%
if $N>0$, and 
\begin{equation}
\Lambda _{n}^{0}=-r_{n-1}=n-\sqrt{-cn}+\frac{2\alpha -2c-1}{4}+\mathcal{O}%
(n^{-1/2}),\qquad n\rightarrow \infty. \label{AsympLn0}
\end{equation}
\end{proposition}

\begin{proof}
We deduce this result from the asymptotic expansions derived before:
combining (\ref{Perron_Laguerre}) with (\ref{asympF}), we obtain 
\begin{equation*}
\widehat{L}_{n-1}^{\alpha }(c)\widehat{F}_{n-1}^{\alpha }(c)=\frac{1}{2\sqrt{%
-cn}}\,\Gamma (n)\Gamma (n+\alpha )\left( 1+\mathcal{O}(n^{-1})\right) .
\end{equation*}%
as $n\rightarrow \infty $. It can be checked that the term of order $%
\mathcal{O}(n^{-1/2})$, which is to be expected here, is actually equal to $%
0 $. Also,
\begin{equation*}
\begin{aligned} \widehat{L}_{n-1}^{\alpha}(c)^{2}&= \frac{D_{\alpha
,c}}{2\sqrt{-cn}}\, \Gamma(n)^{2}n^{\alpha}e^{4\sqrt{-cn}}\left(
1+\mathcal{O}(n^{-1/2})\right),\\ &=\frac{D_{\alpha ,c}}{2\sqrt{-cn}}\,
\Gamma(n)\Gamma(n+\alpha)e^{4\sqrt{-cn}}\left(
1+\mathcal{O}(n^{-1/2})\right) , 
\end{aligned}
\end{equation*}%
using the fact that $\Gamma (n+\alpha )/\Gamma (n)=n^{\alpha }(1+\mathcal{O}%
(n^{-1}))$, as $n\rightarrow \infty $, and the notation 
\begin{equation*}
D_{\alpha ,c}=\frac{e^{c}(-c)^{-\alpha }}{2\pi }.  
\end{equation*}

Putting everything together and using the results in \cite{DHM}, we have 
\begin{equation*}
\begin{aligned}
\Lambda_{n}^{N}&=-\frac{2\sqrt{-cn}}{1+\mathcal{O}(n^{-1})-ND_{\alpha
,c}e^{4\sqrt{-cn}}(1+\mathcal{O}(n^{-1/2}))}+n\frac{L_{n}^{\alpha
}(c)}{L_{n-1}^{\alpha}(c)}, \\
&=-\frac{2\sqrt{-cn}}{1+\mathcal{O}(n^{-1})-ND_{\alpha
,c}e^{4\sqrt{-cn}}(1+\mathcal{O}(n^{-1/2}))} \\ &+n+\sqrt{-cn}+\frac{2\alpha
-2c-1}{4}+\mathcal{O}(n^{-1/2}), \end{aligned}
\end{equation*}

It is important to observe that if $N>0$, the first term in the previous sum
is exponentially small in $n$, so it does not contribute to the final
result. However, if $N=0 $ the first term does contribute, since the
exponential term is not present, and then we have the difference in sign in
the subleading term given in Proposition \ref{Prop_asympL}.
\end{proof}

\section{Asymptotics for $\protect\widehat{Q}_{n}^{\alpha ,c,N}(z)$}
\label{Sect_asymp_Q}

Using the estimates for $\Lambda _{n}^{N}$ (\ref{AsympLn}) and (\ref{AsympLn0}), we describe first the strong
asymptotics for $\widehat{Q}_{n}^{\alpha ,c,N}(z)$, for $z$ away from the interval of orthogonality  
(outer asymptotics): 

\begin{proposition}
Given fixed values of $N\geq 0$, $\alpha>-1$, $c\in(-\infty,0)$ and $z\in%
\mathbb{C}\setminus[0,\infty)$, as $n\to\infty$, the monic polynomials $%
\widehat{Q}_{n}^{\alpha ,c,N}(z)$ verify the following strong asymptotics 
\begin{equation}
\widehat{Q}_{n}^{\alpha ,c,N}(z) =\frac{(-1)^n n!}{2\sqrt{\pi}}e^{z/2+2\sqrt{%
-nz}}(-z)^{-\frac{\alpha}{2}-\frac{1}{4}}n^{\frac{\alpha}{2}-\frac{3}{4}}
\left(\sqrt{-z}\mp\sqrt{-c}\right)\left(1+\mathcal{O}(n^{-1/2})\right),
\label{strong1}
\end{equation}%
and the following ratio asymptotics with respect to monic Laguerre
polynomials: 
\begin{equation}
\frac{\widehat{Q}_{n}^{\alpha ,c,N}(z)}{\widehat{L}_{n}^{\alpha }(z)}= \frac{%
\sqrt{-z}\mp \sqrt{-c}}{\sqrt{n}}-\frac{(\sqrt{-z}\mp \sqrt{-c})^{2}}{2n}+%
\mathcal{O}(n^{-3/2}).  \label{relat1}
\end{equation}
In both cases, the upper sign corresponds to the case $N>0$ and the lower 
sign to $N=0$.
\end{proposition}

\begin{proof}
We rewrite the connection formula (\ref{connection}) as follows: 
\begin{equation}  \label{connection2}
\widehat{Q}_{n}^{\alpha ,c,N}(z)=\left(1+\Lambda_n^N\frac{\widehat{L}%
_{n-1}^{\alpha}(z)}{\widehat{L}_{n}^{\alpha}(z)}\right) \widehat{L}%
_{n}^{\alpha}(z),
\end{equation}
and we use the information obtained so far. Observe that division by
$\widehat{L}_{n}^{\alpha}(z)$ does not cause any problem, since the zeros of this 
polynomial are contained in $[0,\infty)$. From (\ref{asymppi}), we deduce
\begin{equation*}
\frac{\widehat{L}_{n-1}^{\alpha }(z)}{\widehat{L}_{n}^{\alpha
}(z)}=-\frac{1}{n}\left[1-\sqrt{-\frac{z}{n}}+\frac{1-2z-2\alpha}{4n}+\mathcal{O}(n^{-3/2})\right],
\qquad n\to\infty,
\end{equation*}
and combining this with the asymptotic behavior for $\Lambda_n^N$ given by 
Proposition \ref{Prop_asympL} and (\ref{asymppi}), we obtain (\ref{strong1}). Then 
formula (\ref{relat1}) for the relative asymptotic behavior is a direct consequence 
of (\ref{connection2}).
\end{proof}

We remark that this is consistent with the result in \cite[Proposition 3.4 a)%
]{Fej}, taking $N=0$ and $M=1$.


Applying the connection formula again, but with the inner asymptotic expansion for Laguerre 
polynomials, we can obtain strong asymptotics of $\widehat{Q}_{n}^{\alpha ,c,N}(x)$ 
for $x\in(0,\infty)$. 
\begin{proposition}
Given fixed values of $N\geq 0$, $\alpha>-1$, $c\in(-\infty,0)$ and $x$ in compact intervals of 
$(0,\infty)$, as $n\to\infty$, the monic polynomials $%
\widehat{Q}_{n}^{\alpha ,c,N}(x)$ verify the following strong asymptotics 
\begin{equation}
\widehat{Q}_{n}^{\alpha ,c,N}(x) =(-1)^{n+1} n!\,
\frac{n^{\alpha/2-3/4} e^{x/2}}{\sqrt{\pi}\, x^{\alpha/2+1/4}}
\left[
\sqrt{x}\,\sin\theta_n^{\alpha}(x)\pm \sqrt{-c}\,\cos\theta_n^{\alpha}(x)
+\mathcal{O}(n^{-1/2})\right],
\label{strong_inner}
\end{equation}%
where the phase function is
\begin{equation}\label{theta_n}
\theta_n^{\alpha}(x)=2\sqrt{nx}-\left(\frac{\alpha}{2}+\frac{1}{4}\right)\pi,
\end{equation}
and again the upper sign corresponds to the case $N>0$ and the lower 
sign to $N=0$.
\end{proposition}

\begin{proof}
We rewrite (\ref{connection}) as follows:
\begin{equation}\label{connection3}
\widehat{Q}_{n}^{\alpha ,c,N}(x) =
\left(1-\frac{\Lambda_n^N}{n}\right)
\widehat{L}_{n}^{\alpha}(x) 
+\frac{\Lambda_n^N}{n}
\widehat{L}_{n}^{\alpha-1}(x),
\end{equation}
where we have used the following identity for standard Laguerre polynomials:
\begin{equation}\label{connection_Laguerre}
L^{\alpha}_{n-1}(z)=L^{\alpha}_n(z)-L^{\alpha-1}_n(z),
\end{equation}
see for example \cite[18.9.13]{DLMF}. Then, we use the classical Fej\'er formula for Laguerre polynomials, 
see \cite[Theorems 8.22.1, 8.22.2]{Szego}, adapted to the monic case:
\begin{equation}\label{Fejer_Laguerre}
\widehat{L}_{n}^{\alpha}(x)=
(-1)^n n!\, \frac{n^{\alpha/2-1/4} e^{x/2}}{\sqrt{\pi}\, x^{\alpha/2+1/4}}
\left[
\cos\theta_n^{\alpha}(x)+\mathcal{O}(n^{-1/2})\right],
\end{equation}
valid for $x$ in compact intervals of $(0,\infty)$, with phase function (\ref{theta_n}).

From the asymptotic expansion of $\Lambda_n^N$ we deduce that
\begin{equation*}
1-\frac{\Lambda_n^N}{n}=\mp\sqrt{\frac{-c}{n}}+\mathcal{O}(n^{-1}), \qquad 
\frac{\Lambda_n^N}{n}=1\pm\sqrt{\frac{-c}{n}}+\mathcal{O}(n^{-1}),
\end{equation*}
where the upper sign corresponds to the case $N>0$ and the lower sign to $N=0$. 
Using this information and (\ref{Fejer_Laguerre}), we expand as $n\to\infty$ 
in (\ref{connection3}), bearing in mind that from the definition 
$\cos\theta_n^{\alpha-1}(x)=-\sin\theta_n^{\alpha}(x)$, and we arrive at the result.
\end{proof}

Another useful asymptotic behavior that one can find in the literature is
the Mehler-Heine type formulas. Is very well known that for $j\in \mathbb{N}%
\cup \left\{ 0\right\}$, the standard Laguerre polynomials verify 
\begin{equation}
\lim_{n\rightarrow \infty }\frac{L_{n}^{(\alpha )}\left( z/(n+j)\right) }{%
n^{\alpha }}=z^{-\alpha /2}J_{\alpha }\left( 2\sqrt{z}\right) ,
\label{MHLag}
\end{equation}%
uniformly for $z$ in compact subsets of $\mathbb{C}$, see \cite[Theorem 8.1.3%
]{Szego}, where $J_{\alpha }$ is the Bessel function of the first kind, and the square root 
takes its principal value. Using this result, we can prove the following:

\begin{proposition}
For given values of $N\geq 0$, $\alpha>-1$, $c\in(-\infty,0)$ and $z$ in
compact subsets of $\mathbb{C}$, the polynomials $\widehat{Q}_{n}^{\alpha
,c,N}(z)$ verify 
\begin{equation}
\lim_{n\rightarrow \infty } \frac{(-1)^n}{n!}\,
\frac{\widehat{Q}_n^{\alpha ,c,N}(z/n)}{n^{\alpha-1/2}} =\mp\sqrt{-c}\,J_{\alpha }\left( 2%
\sqrt{z}\right),  \label{MHQ}
\end{equation}
where the upper sign corresponds to the case $N>0$ and the lower 
sign to $N=0$.
\end{proposition}

\begin{proof}
In order to prove (\ref{MHQ}), we start with the connection formula (\ref%
{connection}) for monic polynomials. In terms of standard Laguerre
polynomials, recall (\ref{PL}), we have 
\begin{equation*}  
\begin{aligned} \frac{(-1)^n}{n!}\,\widehat{Q}_{n}^{\alpha}(z/n) &=
L^{\alpha}_n(z/n)-\frac{\Lambda_n^N}{n} L^{\alpha}_{n-1}(z/n)\\
&=\left(1-\frac{\Lambda_n^N}{n}\right)L^{\alpha}_n(z/n)+
\frac{\Lambda_n^N}{n}L^{\alpha-1}_n(z/n), \end{aligned}
\end{equation*}
where we have used (\ref{connection_Laguerre}) again. Consequently,
\begin{equation}  \label{For2_MH}
\frac{(-1)^n}{n!}\,\frac{\widehat{Q}_{n}^{\alpha}(z/n)}{n^{\alpha}}
=\left(1-\frac{\Lambda_n^N}{n}\right)\frac{L^{\alpha}_n(z/n)}{n^{\alpha}}+
\frac{\Lambda_n^N}{n^2}\frac{L^{\alpha-1}_n(z/n)}{n^{\alpha-1}}.
\end{equation}

Next, using the asymptotic expansion for $\Lambda_n^N$, we deduce that as $%
n\to\infty$, 
\begin{equation*}
1-\frac{\Lambda_n^N}{n}=\mp\sqrt{\frac{-c}{n}}+\mathcal{O}(n^{-1}), \qquad 
\frac{\Lambda_n^N}{n^2}=\mathcal{O}(n^{-1}),
\end{equation*}
where the upper sign corresponds to the case $N>0$ and the lower sign to $N=0$%
. Then multiplication by $n^{1/2}$ in (\ref{For2_MH}) and the use of the
Mehler--Heine asymptotics for Laguerre polynomials (\ref{MHLag}) give the result.
\end{proof}

Again, we note that this is consistent with \cite[Proposition 3.4 b)]{Fej},
taking $N=0$ and $M=1$.


\section{Asymptotics for recurrence coefficients}
\label{Sect_asymp_rc}

The monic Laguerre polynomials satisfy the three-term recurrence relation (%
\ref{TTRRLaguerre}), and from that we can obtain an analogous recurrence for
the perturbed polynomials $\widehat{Q}_n^{\alpha,c,N}(z)$: 

\begin{theorem}
The polynomials $\widehat{Q}_{n}^{\alpha ,c,N}(z)$ satisfy a three term
recurrence relation 
\begin{equation}
\widehat{Q}_{n+1}^{\alpha ,c,N}(z)=(x-\tilde{\beta}_{n})\widehat{Q}%
_{n}^{\alpha ,c,N}(z)-\tilde{\gamma}_{n}\widehat{Q}_{n-1}^{\alpha ,c,N}(z),
\label{TTRRQ}
\end{equation}%
where the coefficients are given by 
\begin{align*}
\tilde{\beta}_{n} &=\beta_n + \Lambda^N_n - \Lambda^N_{n+1}, \\
\tilde{\gamma}_{n} &= \frac{\Lambda^N_n}{\Lambda^N_{n-1}} \gamma_{n-1},
\end{align*}
and $\beta_n$ and $\gamma_n$ are given by (\ref{bngnLaguerre}).
\end{theorem}

\begin{proof}
Let us denote by $\tilde \beta$ and $\tilde \gamma$ the coefficients of the
three term recurrence relation for the polynomials $\widehat Q_{n}^{\alpha
,c,N}(z)$: 
\begin{equation*}
x \widehat Q_{n}^{\alpha ,c,N}(z) = \widehat Q_{n+1}^{\alpha ,c,N}(z) +
\tilde\beta_n \widehat Q_{n}^{\alpha ,c,N}(z) + \tilde\gamma_n \widehat
Q_{n-1}^{\alpha ,c,N}(z)
\end{equation*}
Using the connection formula \eqref{connection} on both sides of the
previous equation we get 
\begin{multline*}
x\widehat{L}_{n}^{\alpha }(z)+\Lambda _{n}^{N}x\widehat{L}%
_{n-1}^{\alpha}(z)= \widehat{L}_{n+1}^{\alpha }(z)+(\Lambda _{n+1}^{N} +
\tilde \beta_n) \widehat{L}_{n}^{\alpha }(z) \\
+(\tilde \beta_n \Lambda _{n}^{N} + \tilde \gamma_n ) \widehat{L}%
_{n-1}^{\alpha}(z) + \gamma_n \Lambda _{n-1}^{N}\widehat{L}%
_{n-2}^{\alpha}(z).
\end{multline*}
We use the three-term recurrence relation for the Laguerre polynomials %
\eqref{TTRRLaguerre} on the left hand side of the previous equation and we
obtain 
\begin{multline*}
\widehat L_{n+1}^{\alpha }(z)+(\beta _{n} + \Lambda _{n}^{N}) \widehat
L_{n}^{\alpha}(z)+(\gamma _{n} + \Lambda _{n}^{N}\beta _{n-1}) \widehat
L_{n-1}^{\alpha }(z) + \Lambda _{n}^{N}\gamma _{n-1} \widehat
L_{n-2}^{\alpha }(z) \\
= \widehat{L}_{n+1}^{\alpha }(z)+(\Lambda _{n+1}^{N} + \tilde \beta_n) 
\widehat{L}_{n}^{\alpha }(z) +(\tilde \beta_n \Lambda _{n}^{N} + \tilde
\gamma_n ) \widehat{L}_{n-1}^{\alpha}(z) + \gamma_n \Lambda _{n-1}^{N}%
\widehat{L}_{n-2}^{\alpha}(z).
\end{multline*}
Since the Laguerre polynomials are a basis for the space of polynomials, we
obtain the following equations 
\begin{align}
\beta _{n} + \Lambda _{n}^{N} &= \Lambda _{n+1}^{N} + \tilde \beta_n,
\label{eq:three_t_1} \\
\gamma _{n} + \Lambda _{n}^{N}\beta _{n-1} &= \tilde \beta_n \Lambda
_{n}^{N} + \tilde \gamma_n,  \label{eq:three_t_2} \\
\Lambda _{n}^{N}\gamma _{n-1} &= \tilde \gamma_n \Lambda _{n-1}^{N}.
\label{eq:three_t_3}
\end{align}
The theorem follows directly from formulas \eqref{eq:three_t_1} and %
\eqref{eq:three_t_3}.
\end{proof}

\begin{remark}
From (\ref{eq:three_t_2}), substituting the expressions for $\tilde \gamma_n$
and $\tilde \beta_n$, we obtain the following non--linear recursion for $%
\Lambda_n^N$: 
\begin{equation}  \label{eq:recurrence_Lambda}
\Lambda_{n+1}^N-\Lambda_{n}^N=-\frac{\gamma_n}{\Lambda_{n}^N}+\frac{%
\gamma_{n-1}}{\Lambda_{n-1}^N}+2
\end{equation}
It follows directly from \eqref{eq:recurrence_Lambda} that 
\begin{equation*}
\Lambda^N_{n+1}-\Lambda_2^N=\sum_{i=3}^{n+1} (\Lambda^N_{i}-\Lambda^N_{i-1})
=-\frac{\gamma_n}{\Lambda_n}+2(n-1)+\frac{\gamma_1}{\Lambda_1}.
\end{equation*}
This gives the following recursion for $\Lambda_n$: 
\begin{equation}  \label{eq:recurrence_Lambda_2}
\Lambda^N_{n+1}=-\frac{n(n+\alpha)}{\Lambda^N_n}+2(n-1)+\frac{\gamma_1}{%
\Lambda^N_1}+\Lambda_2^N.
\end{equation}
If we define $\Lambda_n^N=\varrho_n^N/\varrho_{n-1}^N$, then (\ref%
{eq:recurrence_Lambda_2}) becomes 
\begin{equation}  \label{eq:recurrence_rho}
\varrho_{n+1}^N-\left(2(n-1)+\frac{\gamma_1}{\Lambda^N_1}+\Lambda_2^N\right)%
\varrho_n^N +n(n+\alpha)\varrho_{n-1}^N=0.
\end{equation}
\end{remark}

From this result and the asymptotic expansion for $\Lambda^N_n$, see (\ref{AsympLn}) 
and (\ref{AsympLn0}), it is straightforward to deduce the large $n$ asymptotic behavior of the 
recurrence coefficients $\tilde{\beta}_n$ and $\tilde{\gamma}_n$, in terms of the 
recurrence coefficients for monic Laguerre polynomials, $\beta_n$ and $\gamma_n$, 
given by (\ref{bngnLaguerre}):
\begin{corollary} 
As $n\to\infty$, for fixed $c\in(-\infty,0)$ and $N\geq 0$, the coefficients $\tilde{\beta}_n$ and 
$\tilde{\gamma}_n$ of the three term 
recurrence relation for monic Laguerre--Geronimus orthogonal polynomials satisfy
\begin{equation*}
\begin{aligned}
\tilde{\beta}_n&=\left(1-\frac{1}{2n}\mp\frac{\sqrt{-c}}{4n^{3/2}}+\mathcal{O}(n^{-2})\right)\beta_n,\\
\tilde{\gamma}_n&=\left(1+\frac{1}{n}\mp\frac{\sqrt{-c}}{2n^{3/2}}+\mathcal{O}(n^{-2})\right)\gamma_{n-1},
\end{aligned}
\end{equation*}
where the upper sign corresponds to the case $N>0$ and the lower sign to the case $N=0$.
\end{corollary}

\section{Hypergeometric representation of $\widehat{Q}_{n}^{\protect\alpha ,c,N}(z)$}
\label{Sect_hypergeom}

In this section we will derive a representation of the Geronimus
perturbed family of orthogonal polynomials as hypergeometric functions.
For this we need the connection formula \eqref{connection} together with the hypergeometric
representation of the monic Laguerre polynomials, that can be obtained from \cite[18.5.12]{DLMF}:
\begin{align}
\widehat{L}_{n}^{\alpha }\left(z\right) &=\frac{\left( -1\right)
^{n}\Gamma \left( n+\alpha +1\right) }{\Gamma \left( \alpha +1\right) }%
\pFq{1}{1}{-n}{\alpha+1}{z} \nonumber\\
&=\left( -1\right) ^{n}(\alpha +1)_{n}\sum_{k=0}^{\infty }\frac{\left(
-n\right) _{k}}{\left( \alpha +1\right) _{k}}\frac{z^{k}}{k!},
\label{eq:hyperg_laguerre}
\end{align}
\begin{theorem}
\label{thm:expresion_hyperg}
The monic polynomials $\widehat{Q}_{n}^{\alpha ,c,N}(z)$ have the following 
hypergeometric representation
$$\widehat{Q}_{n}^{\alpha ,c,N}(z)= C_{n,\alpha} \,\, \pFq{2}{2}{-n,\,1+e^N_n}{\alpha +1,\, e^N_n}{z},$$
where
$$C_{n,\alpha } =\left( 1-\frac{\Lambda _{n}^{N}}{n+\alpha }\right) \left(
-1\right) ^{n}(\alpha +1)_{n}\,, \quad e^N_n=\frac{n(n+\alpha -\Lambda _{n}^{N})}{\Lambda _{n}^{N}}.$$
\end{theorem}

\begin{proof}
It follows from the connection formula \eqref{connection} and \eqref{eq:hyperg_laguerre} that 
\begin{multline*}
\widehat{Q}_{n}^{\alpha ,c,N}(z) =\left( -1\right) ^{n}(\alpha
+1)_{n}\sum_{k=0}^{\infty }\frac{\left( -n\right) _{k}}{\left( \alpha
+1\right) _{k}}\frac{z^{k}}{k!} \\
+\Lambda _{n}^{N}\,\left( -1\right) ^{n-1}(\alpha
+1)_{n-1}\sum_{k=0}^{\infty }\frac{\left( -n+1\right) _{k}}{\left( \alpha
+1\right) _{k}}\frac{z^{k}}{k!}.
\end{multline*}%
By a straightforward calculation, $\widehat{Q}_{n}^{\alpha ,c,N}(z)$ can be written as
\begin{equation}
\label{eq:Qn_hiperg_brack}
\widehat{Q}_{n}^{\alpha ,c,N}(z)=\left( -1\right) ^{n}(\alpha +1)_{n}\sum_{k=0}^{\infty }\left( \frac{%
\left( -n\right) _{k}}{\left( \alpha +1\right) _{k}}\frac{z^{k}}{k!}
\left[ 1-\frac{(n-k)\Lambda _{n}^{N}}{n(\alpha +n)}\right] \right).
\end{equation}
Next, we rewrite the expression in square brackets as 
\begin{equation}
\label{eq:rational_brack}
1-\frac{(k-n)\Lambda _{n}^{N}}{(-n)(\alpha +n)}=\frac{\Lambda^N_n}{n(\alpha+n)}\left( k+e^N_n\right)= \frac{\Lambda^N_n}{n(\alpha+n)} \, e^N_n\frac{(1+e^N_n)_k}{(e^N_n)_k}, 
\end{equation}
where
$$e^N_n=\frac{n(n+\alpha-\Lambda^N_n)}{\Lambda^N_n}.$$
By replacing \eqref{eq:rational_brack} into \eqref{eq:Qn_hiperg_brack}, we obtain
\begin{align*}
\widehat{Q}_{n}^{\alpha ,c,N}(z)&=\frac{\left( -1\right) ^{n}(\alpha +1)_{n} \, \Lambda^N_n \, e^N_n}{n(\alpha+n)}   \, \sum\limits_{k=0}^{\infty }
\dfrac{(-n)_{k}}{\left( \alpha +1\right) _{k}}  \frac{(1+e^N_n)_k}{(e^N_n)_k} \, \dfrac{z^{k}}{k!} \\
&=\left( 1-\frac{\Lambda _{n}^{N}}{n+\alpha }\right) \left(
-1\right) ^{n}(\alpha +1)_{n}\, \pFq{2}{2}{-n,\, 1+e^N_n}{\alpha+1,\, e^N_n}{z}.
\end{align*}
This completes the proof of the theorem.
\end{proof}

\begin{remark}
The hypergeometric functions ${}_{2}F_{2}$ are solutions to a third-order differential equation \cite[16.8.3]{DLMF}. Therefore, Theorem \ref{thm:expresion_hyperg} implies that the perturbed polynomials $\widehat{Q}_{n}^{\alpha ,c,N}(z)$ are solutions to
\begin{equation}
\label{eq:third_order_ODE}
z^2 y'''-x(x-e^N_{n}-\alpha-2)y''-( (e^N_{n}-n+2)z-(\alpha+1)e^N_{n})y'+n(e^N_{n}+1)y =0.
\end{equation}
This differential equation can be easily obtained from the holonomic equation for the polynomials $\widehat{Q}_{n}^{\alpha ,c,N}$ (see \cite[Section 4.1]{BHR-LNCS14}):
\begin{equation}
\label{eq:second_order_ODE}
y''+R(z)y'+S(z)y=0,
\end{equation}
where
\begin{align*}
 R(z)&=-\frac{\Lambda^N_{n}}{z\Lambda_{n}^N+(n-\Lambda_{n}^N)(n+\alpha-\Lambda_{n}^N)}+\frac{\alpha+1}{z}-1,\\
 S(z)&=\frac{z\Lambda^N_{n}+(n-\Lambda_{n}^N)(n+\alpha)}{z(z\Lambda_{n}^N+(n-\Lambda_{n}^N)(n+\alpha-\Lambda_{n}^N))}+\frac{n-1}{z}.
\end{align*}
Indeed, \eqref{eq:third_order_ODE} is obtained by multiplying the derivative of \eqref{eq:second_order_ODE} by $z^2$ and adding it to \eqref{eq:second_order_ODE} multiplied by $-z(z-e^N_{n}-\alpha-2)-z^2R(z)$.

\end{remark}

\section*{Acknowledgements}
The financial support from the project MTM2012-36732-C03-01 (A. Dea\~{n}o and E. J. Huertas) and the project MTM2012--34787 (A. Dea\~{n}o), both from Ministerio de Econom\'ia y Competitividad of Spain, is gratefully acknowledged.

P. Rom\'an was partially supported by CONICET grant PIP 112-200801-01533 and by SeCyT-UNC.

The authors thank Paco Marcell\'an (Universidad Carlos III, Madrid) for useful and stimulating discussions on the topic and scope of this paper.

\end{document}